\documentclass{article}

	\usepackage{amsmath}
	\usepackage{amssymb}
	\usepackage{amsthm}
	\usepackage[colorlinks=true, urlcolor=blue]{hyperref}
	
	\newtheorem{proposition}{Proposition}[section]
	\newtheorem{lemma}[proposition]{Lemma}
	\newtheorem{definition}[proposition]{Definition}
	\newtheorem{remark}[proposition]{Remark}
	\newtheorem{example}[proposition]{Example}

	\newcommand{\rk}{\operatorname{rk}}
	\newcommand{\Sy}{\operatorname{Sym}}
	\newcommand{\Ker}{\operatorname{Ker}}
	\newcommand{\Ork}{\operatorname{Ork}}
	
	\newcommand{\Ps}{\mathbb{P}}
	
	\newcommand{\Span}[1]{\left\langle\,#1\,\right\rangle}
	\newcommand{\blen}{\operatorname{b\ell}}
	\newcommand{\crk}{\operatorname{crk}}
	\newcommand{\al}{\operatorname{al}}
	
\begin{document}
\title{Generic Power Sum Decompositions\\ and Bounds for the Waring Rank}
\author{Edoardo Ballico\\Dipartimento di Matematica, Universit\`a di Trento, Povo (Italy)\\\\Alessandro De Paris\\Dipartimento di Matematica e Applicazioni ``Renato Caccioppoli'',\\Universit\`a degli Studi di Napoli Fe\-de\-ri\-co II (Italy), deparis@unina.it\\\\Accepted version of an article published in\\ Discrete Comput. Geom. (2017), 57(4), 896--914,\\ DOI: 10.1007/s00454-017-9886-7.\\ \href{http://link.springer.com/article/10.1007/s00454-017-9886-7}{The final publication is available at link.springer.com}.}
\date{}
\maketitle

\begin{abstract}
A notion of \emph{open rank}, related with generic power sum decompositions of forms, has recently been introduced in the literature. The main result here is that the maximum open rank for plane quartics is eight. In particular, this gives the first example of $n,d$, such that the maximum open rank for degree $d$ forms that essentially depend on $n$ variables is strictly greater than the maximum rank. On one hand, the result allows to improve the previously known bounds on open rank, but on the other hand indicates that such bounds are likely quite relaxed. Nevertheless, some of the preparatory results are of independent interest, and still may provide useful information in connection with the problem of finding the maximum rank for the set of all forms of given degree and number of variables. For instance, we get that every ternary forms of degree $d\ge 3$ can be annihilated by the product of $d-1$ pairwise independent linear forms.\\
Keywords: Power sum , Waring rank , tensor rank , symmetric tensor.\\
MSC 2010: 15A21, 15A69, 15A72, 14A25, 14N05, 14N15.\\
\end{abstract}

\section*{Acknowledgements}
Some improvements are due to an anonymous referee, in particular Example~\ref{SE} and a simplification in the proof of Lemma~\ref{Distinti}.

Financial support by MIUR (IT), GNSAGA of INdAM (IT) and Universit\`a degli Studi di Napoli Federico II (IT).

\section{Introduction}

Part of the considerable amount of work that the scientific community is devoting to recently emerged aspects of tensor theory addresses the polynomial Waring problem (see \cite{L}). If in a power sum decomposition
\begin{equation}\label{Wd}
f={l_1}^d+\cdots +{l_r}^d\;
\end{equation}
of a degree $d$ homogeneous polynomial, the number $r$ of summands is the minimum possible, then $r=:\rk f$ is the \emph{(Waring) rank} of $f$, and \eqref{Wd} is often called a \emph{Waring decomposition}. When $f$ can be regarded as a symmetric tensor (in particular, when the coefficients are in a field of characteristic zero), the Waring rank becomes the \emph{symmetric rank} (\footnote{Sometimes the term ``Waring decomposition'' has been used to indicate simply a power sum decomposition, without the minimality hypothesis. We also mention that the symmetric rank is sometimes called \emph{polar rank}: see \cite{P}.}). Perhaps, in its broadest sense, the polynomial Waring problem consists of finding the rank of specified polynomials (see \cite[Introduction]{CCG}). In a restricted sense, and in analogy with the number-theoretic Waring problem, one wonders about the maximal rank of homogeneous polynomials of fixed degree and number of variables (see \cite{G}). The solution to the main (`generic') version of such a problem, given in \cite{AH}, is now a classical result. But the maximum rank $\rk(n,d)$ for \emph{all} homogeneous polynomials of degree $d$ and number $n$ of variables, at the time of writing, is known only when $n\le 2$, when $d\le 2$, and in a few special cases: for $n=3$ we have $\rk(3,3)=5$ (see \cite[Chap.~2]{K} or \cite{LT}), $\rk(3,4)=7$ (see \cite[Chap.~3]{K} or \cite{D}) and $\rk(3,5)=10$ (see \cite{BuT} and \cite{D2}); see \cite{BuT} for more details.

To improve our knowledge on $\rk(n,d)$, we may exploit outcomes of careful investigations on polynomials of low dimension and degree (such as those in \cite{K}, \cite{LT}, \cite{BGI}, \cite{D}, \cite{D2}), or on some classes of polynomials of special interest (as done in~\cite{CCG}). One may also look for general bounds, as in \cite{J}, which improves a bound given earlier in~\cite{BBS}, \cite{BBS2}, by means of a modified version of the inductive procedure involved. Since the rank is not well-behaved in view of the inductive steps in~\cite{BBS} and~\cite{J}, the authors of the mentioned papers introduce auxiliary notions of rank, based on the minimization over decompositions that are `sufficiently generic' in some sense. They also need to provisionally focus on forms that \emph{essentially depend} on $n$ variables. For more details, see \cite[Introduction]{J}. In particular, in \cite[Def.~2]{J}, the \emph{open (Waring) rank} $\Ork(F)$ of a form $F$, and the maximum open rank $\Ork(n,d)$ for degree $d$ forms that essentially depend on $n$ variables, are introduced. With these notions, one can give estimates on $\Ork(n,d)$ by induction, and use the obvious inequality $\rk(n,d)\le\Ork(n,d)$ to get estimates on the Waring rank. Moreover, at the time of publishing of \cite{J}, no pairs $(n,d)$ with $\rk(n,d)\ne\Ork(n,d)$ were known, so that the general validity of the equality $\rk(n,d)=\Ork(n,d)$ could not have been excluded.

Successively, the bound given in \cite[Corollary~6]{J} has been drastically improved by \cite[Corollary~9]{BT}, by means of a simple argument. The improvement also implies that, in the hypothesis that $\rk(n,d)=\Ork(n,d)$ holds in general, then the inductive bound given in \cite[Theorem~4]{J} can not be sharp (this is less surprising, because the improvement given by \cite[Corollary~6]{J} with respect to the earlier results in \cite{BBS} is precisely based on the failing of sharpness of \cite[Theorem~4]{J} for $(n,d)=(3,3)$).

In spite of such weakness, we believe that investigations based on the open rank still are of use. To say the least, they convey attention on some interesting aspects of tensor rank theory. Furthermore, open rank is about generic power sum decompositions, that are of their own interest (the fact that they are well-behaved with respect to induction on $n,d$ gives an indication). Note also that \cite[Corollary~6]{J} does not rely on Alexander-Hirschowitz theorem, so that the overall complexity of its proof is actually much smaller than that of \cite[Corollary~9]{BT}.

In this paper we show that that $\Ork(3,4)=8$. Hence, we get an example where $\Ork(n,d)\ne\rk(n,d)$ (for this part of the story the inequality $\Ork(3,4)\ge 8$ suffices, and it is a consequence of Example~\ref{Esempio}). The inequality $\Ork(3,4)\le 8$ is worked out in Section~\ref{Tq}, and requires considerably more work, but we believe that has some interest as well. For instance, note that the inductive procedure based on degree $3$ gives a bound of $9$ in degree $4$. We also mention that a condition considered while the present article was in preparation, has successively played a nontrivial role in \cite{D2}, where the sharp upper bound $\rk(3,5)\le 10$ has been found. We also believe that Proposition~\ref{Rette} (a development of \cite[Proposition~4.1]{D}) deserves some interest.

\section{Preparation}

All vector spaces will be understood over a fixed algebraically closed field $\mathbb{K}$. For simplicity of exposition, we also assume $\operatorname{char }\mathbb{K}=0$ (though the results hold under more general hypotheses; e.g., when dealing with a degree $d$ form, it might often be assumed $\operatorname{char }\mathbb{K}>d$, as in \cite{BBS}). A projective space $\Ps V$ is understood as the set of all one-dimensional subspaces $\Span{v}$, $v\ne 0$, of the vector space $V$. When the scheme structure is needed, $\Ps V$ has to be replaced by $\operatorname{Proj}\Sy^\bullet V^\ast$.

Throughout the paper $S_\bullet=\Sy^\bullet S_1$, $S^\bullet=\Sy^\bullet S^1$ will denote standard graded rings, dually paired by \emph{apolarity}, that is a perfect pairing naturally induced by a fixed perfect pairing between $S^1$ and $S_1$ (see \cite[Introduction]{D}). Apolarity may also be viewed as a particular case of tensor contraction and, conversely, contraction of forms in $S_\bullet$ by forms in $S^\bullet$ can be defined in terms of apolarity (see again \cite[Introduction]{D} for more details). When dual bases
\[
x_0,\ldots ,x_n\in S_1\;,\qquad x^0,\ldots ,x^n\in S^1
\]
are fixed, the contraction of $f\in S_\bullet$ by $x^i$ is simply the partial derivation with respect to $x_i$. By this reason, for all $p\in S^\bullet$ we denote by
\[
\partial_p:S_\bullet\to S_\bullet\;,
\]
the contraction by $p$ operator. We also have that $p\mapsto\partial_p$ is a linear operation and $\partial_{pq}=\partial_p\circ\partial_q$. This allows to identify $S^\bullet$ with the ring of constant coefficients linear differential operators on $S_\bullet$ (apolarity is often directly defined by means of this property). It is convenient to explicitly mention that for a \emph{linear} form $t\in S^1$ and all $f,g\in S_\bullet$, we have
\[
\partial_t(fg)=(\partial_tf)g+f\partial_tg\;.
\]

The sign $\perp$ will refer to orthogonality with respect to the perfect pairing \emph{in fixed degree}; we shall not use it to denote apolar ideals. The \emph{partial polarization map} $f_{\delta,d}:S^\delta\to S_d$ of $f\in S_{d+\delta}$, is given by $f_{\delta,d}(t):=\partial_tf$.

The (Waring) rank of $f\in S_d$, denoted by $\rk f$, is the minimum of all nonnegative integers $r$ for which there exists a decomposition $f={v_1}^d+\cdots +{v_r}^d$ with $\Span{v_1},\ldots ,\Span{v_r}\in\Ps S_1$. This can also be regarded as a particular instance of a more general notion of rank of a point with respect to an arbitrary variety in a projective space (see \cite[5.2.1]{L}). Let us also rephrase below the definition of open rank given in \cite[Definition~2]{J}.

\begin{definition}
The \emph{open (Waring) rank} of $f\in S_d$, denoted by $\Ork f$, is the minimum of all nonnegative integers $r$ with the following property: for every Zariski closed, proper subset $X\subsetneq\Ps S_1$, there exists a decomposition
\[
f={v_1}^d+\cdots +{v_r}^d
\]
with $\Span{v_1},\ldots ,\Span{v_r}\in\Ps S_1\setminus X$. The minimum $r$ for a fixed $X$ is denoted by $\Ork(f,X)$ (in particular, $\rk f=\Ork(f,\emptyset)$).
\end{definition}

Loosely speaking, the open rank of $f$ is the least number of summands for which $f$ admits a generic power sum decomposition.

\section{Open rank in comparison with rank}

An obvious relationship between rank and open rank is $\rk f\le\Ork f$. Moreover, in the many cases where $f$ admits an essentially unique Waring decomposition, we have $\rk f<\Ork f$.

On one hand, open rank may have its own interest, at least from a theoretical viewpoint. Even from the applicative viewpoint, since Waring decompositions are related with tensor decompositions (which have many applications: see \cite{L}), it can not be excluded that for some purposes one might want to exclude decompositions of some special kind. On the other hand, since $\rk f\ne\Ork f$ in many cases, one would not expect that the open rank can give information on rank. But, as a matter of facts, to some extent it can, as shown by a simple result about binary forms we are going to state. 

Let us first recall that when $\dim S_1=2$, the \emph{length} of $f\in S_d$, which we denote by $\ell(f)$, is the least of all $s$ for which there exists a nonzero $l\in S^s$ such that $\partial_lf=0$. In other words, it is the initial degree of the \emph{apolar ideal} of $f$, $I_f:=\left\{x\in S^\bullet:\partial_xf=0\right\}$ (see \cite[Def.~1.32 and Lemma~1.33]{IK}). The notion of length of a binary form can be generalized in various ways for forms in more indeterminates: see \cite[Def.~5.66]{IK}. Nowadays, terms related to length are replaced by similar terms related with rank, probably because of the renewed interest in the interplay with the rank of tensors. For instance, according to \cite[Theorem 1.44]{IK}, for binary forms the length coincide with the border rank.

\begin{proposition}\label{SOr}
Assume $\dim S_1=2$. For all nonzero $f\in S_d$ we have
\[
\Ork f=d+2-\ell(f)\;.
\]
\end{proposition}
\begin{proof}
We have $2 \le 2b \le d+2$, with $b:=\ell(f)$. Let $P:=\Span f\in\Ps S_d$ and $C$ be the rational normal curve given by $d$-th powers in $\Ps S_d$. For any integer $k$, let $\mathcal {Z}(P,k)$ be the set of all degree $k$ zero-dimensional schemes $W\subset C$ such that $P\in \langle W\rangle$ (scheme-theoretic, projective span). The set $\mathcal {Z}(P,k)$ is naturally identified with the projective space associated to the degree $k$ component of the ideal $I_f$, apolar to $f$. The Artin graded algebra $A_f = S^\bullet/I_f$ is a complete intersection, with $I_f$ generated by a form of degree $b$ and a form of degree $d+2-b$ (\cite[Theorem 1.44]{IK}). $\mathcal{Z}(P,k) =\emptyset$ if $k<b$. If $2b=d+2$, then $\rk_CP =b=d+2-b$ and hence $\Ork f\ge b=d+2-b$.

Now assume $2b \ne d+2$. In this case $\mathcal {Z}(P,b)$ has a unique element, $Z$,
and if $b < k < d+2-b$, then each element of $\mathcal {Z}(P,k)$ is the union of $Z$ and a scheme $E\subset C$ of degree $k-b$. Hence $\Ork f \ge d+2-b$ in this case, too.

To prove the opposite inequality $\Ork f \le d+2-b$, it suffices to prove that the linear system of divisors on $\Ps S_1$ given by $\left(I_f\right)_{d+2-b}$ has no base points. But this immediately follows from the fact that $I_f$ is the ideal of a complete intersection, generated by a form of degree $b$ and a form of degree $d+2-b\ge b$.\hfill $\square$
\end{proof}

\begin{remark}\label{SOrg}\rm
Every rational normal curve in a $d$-dimensional projective space $\Ps^d$ corresponds to the curve given by $d$-th powers in $\Ps S_d$ through some isomorphism of projective spaces. Hence Proposition~\ref{SOr} holds as well if we consider a point $P\in\Ps^d$ in place of $f$, its open rank with respect to a rational normal curve $\gamma\subseteq\Ps^d$ and its border rank with respect to $\gamma$ in place of length (taking into account \cite[Theorem 1.44]{IK}).
\end{remark}

From the Comas-Seiguer theorem (see \cite[Theorem~11]{CS} or \cite[Theorem~9.2.2.1]{L}) immediately follows that $\Ork{f}=\rk f$ for all binary forms $f$ with rank higher than the generic. This fact indicates that the open rank may help to solve the polynomial Waring's problem that asks for the maximum rank for the set of \emph{all} forms of given degree and number of variables. Pursuing this indication, let us consider the maximum rank and the maximum open rank for degree $d$ binary forms: they are, respectively, $d$ and $d+1$. Note also that $\Ork f=d+1$ only for $d$-th powers of linear forms, which can be regarded as forms in only one variable. Hence the maximum rank coincide with the maximum open rank on the set of all \emph{essentially} binary forms. Moreover, Jelisiejew showed in \cite{J} that one can bound the maximum open rank for forms of degree $d$ that essentially depend on $n$ variables, basically following the induction procedure on $n,d$ introduced by Bia\l ynicki-Birula and Schinzel in \cite{BBS}.

On the other hand, the above encouraging features of open rank seem not to suffice for the determination of maximum rank. Indeed, the Jelisiejew's improvement of the bound proved by Bia\l ynicki-Birula and Schinzel, exploit the fact that the induction procedure do not give a sharp upper bound for essentially ternary cubics (neither on rank nor on open rank). Even the Jelisiejew's bound, as a bound on maximum rank, is known to be not sharp for quartics. What is more, below we point out that the equality between maximum rank for ternary quartics and maximum open rank for essentially ternary quartics fails. To this end we shall give a geometric example, with some use of specific results from modern algebraic geometry. A more explicit and genuinely algebraic example, suggested by an anonymous referee, will follow.

To introduce the geometric example, let us recall, in general, that given subschemes $X,Y$ of a projective space $\operatorname{Proj} S^\bullet$, with respective ideals $I(X),I(Y)\subseteq S^\bullet$, the subscheme $Y'$ defined by the ideal \[\left(I(X):I(Y)\right):=\left\{f: fg\in I(X)\;\forall g\in I(Y)\right\}\] always contains $X\setminus Y$ (as sets). When $X$ is reduced we have $Y'=\overline{X\setminus Y}$ as topological spaces, but if, say, $Y$ is a hyperplane defined by a (nonzero) linear form $l$, $I(Y)=\left(l\right)$, and $X$ is its double, $I(X)=\left(l^2\right)$, then we have $Y'=Y$ and $X\setminus Y=\emptyset$. We shall need to consider a case where $Y$ is a line in the plane and $X$ a zero-dimensional scheme (intuitively a set of points, some of them coinciding in a way that is encoded in the scheme structure). In this case, or more generally whenever $Y$ is a hyperplane, $I(Y)=(l)$, the intersection ideal $I(X)\cap I(Y)$ clearly consists of all $lf$ with $f\in\left(I(X):I(Y)\right)$. This gives rise to an exact sequence of graded $S^\bullet$-modules
\[
0\to\left(I(X):I(Y)\right)(-1)\to I(X)\to\frac{I(X)}{I(X)\cap I(Y)}\to 0\;,
\]
where $(-1)$ denotes a degree shift and the first map is the multiplication by~$l$. Taking into account that $I(X)/\left(I(X)\cap I(Y)\right)\cong\left(I(X)+I(Y)\right)/I(Y)$, and passing to sheaves (generalities about the technical procedure can be found, e.g., in \cite[p.~116, Definition]{H}), we get the operative description \cite[Notation~4.3]{BB3} that underlies \cite[Lemma 5.1]{BB3}, a result we are going to use. The scheme $Y'$ can be called the residual scheme of $Y$ to $X\cup Y$ with respect to $\operatorname{Proj} S^\bullet$, according to \cite[Def.~9.2.1]{F}), and we shall use a notation of the form $\mbox{Res}_Y(X)$.

A technical condition of the form $h^1\left(\mathcal{I}_X(d)\right)=0$, involved in the statement of the mentioned lemma, simply amounts to say that $X$ imposes independent conditions to degree $d$ forms. More geometrically, this also means that for whatever subscheme $X'\subsetneq X''$ of $X$, there always exists a degree $d$ form that vanishes on (the whole of) $X'$ but not on $X''$. Residual schemes are also involved in the other auxiliary \cite[Lemma 34]{BGI}, but at the technical core of the example lies a further interesting result in a similar vein, for which we refer to \cite[Remarque (i), p.~116]{EP}.

\begin{example}\label{Esempio}
Assume $\dim S_1=3$ and let $\nu:\Ps S_1\to\Ps S_4$, $\Span{v}\mapsto\Span{v^4}$ be the Veronese embedding. We can certainly fix a degree $4$ curvilinear zero-dimensional scheme $Z\subset\Ps S_1$ (\footnote{We say that a zero-dimensional scheme is \emph{curvilinear} if it can be embedded in some smooth curve.}) such that $\dim\Span{\nu(Z)}=3$ and $\deg (Z\cap L)=3$ for (exactly) one line $L$. Since $\dim\Span{\nu(Z)}=3$, we can fix $P=\Span{f}\in\Ps S_4$, such that $P\in\Span{\nu(Z)}$ and $P\not\in\Span{\nu\left(Z'\right)}$ for every subscheme $Z'\subsetneq Z$. We show that $\Ork f\ge 8$.

Assume that $a:=\Ork f\le 7$. Fix a closed set $X\subsetneq\Ps S_1$ containing the union of the finitely many lines $D$ with $\deg (Z\cap D)\ge 2$. In particular, $X$ contains $L\cup Z_{red}$. Fix a set $B\subset\Ps S_1\setminus X$ such that $\sharp(B)=a$, $P\in\Span{\nu(B)}$ and $P\not\in\Span{\nu\left(B'\right)}$ for all $B'\subsetneq B$. Because of the last condition, $h^1(\mathcal {I}_B(4)) =0$. Therefore, at most $5$ of the points of $B$ are collinear. Since $Z_{red}\subset X$, we have $Z\cap B =\emptyset$, and since $P\in\Span{\nu(B)}\cap\Span{\nu(Z)}$, we have $h^1(\mathcal {I}_{Z\cup B}(4))>0$. Since $\deg (Z\cup B) = 4+a \le 11 < 4\cdot 3$, either there is a line $R\subset\Ps S_1$ such that $\deg (R\cap (Z\cup B)) \ge 6$ or there is a conic $C\subset\Ps S_1$ such that $\deg (C\cap (Z\cup B)) \ge 10$ (see \cite[Remarque (i), p.~116]{EP} and take into account that, according to \cite[p.~112, l.~3]{EP}, by a ``groupe de points'' is meant a zero-dimensional scheme, not necessarily reduced).

First assume the existence of a line $R\subset\Ps S_1$ such that $\deg (R\cap (Z\cup B))\ge 6$. Since $\deg (Z) =4$, $B\cap X =\emptyset$ and $X$ contains each line $D$ with $\deg (Z\cap D)\ge 2$, we have $\deg (Z\cap R)\le 1$. Hence $\sharp (B\cap R) \ge 5$. Since  $h^1(\mathcal {I}_B(4)) =0$, we get $\sharp (B\cap R) =5$, $\deg (Z\cap R) =1$ and $\deg (R\cap (Z\cup B))) = 6$. We have $\deg (\mbox{Res}_R(Z\cup B) ) =a+4-6 \le 5$. Hence either $h^1(\mathcal {I}_{\mbox{Res}_R(Z\cup B)}(3)) =0$ or $a=7$ and there is a line $R'\subset \Ps S_1$ such that $R'\supseteq \mbox{Res}_R(Z\cup B)$ (\cite[Lemma 34]{BGI}). First assume $h^1(\mathcal {I}_{\mbox{Res}_R(Z\cup B)}(3)) =0 $. Since $Z\cap B =\emptyset$, we can exploit \cite[Lemma 5.1]{BB3} and deduce that $Z\cup B \subset R$, because $P$ is in $\Span{\nu(B)}\cap\Span{\nu(Z)}$, but not in the span of subschemes that are smaller than $B$ or smaller than $Z$. Hence $Z$ is contained in a line. Since $\deg (Z\cap L) =3$, we get a contradiction. Now assume $a=7$ and the existence of a line $R'$ such that $R'\supseteq \mbox{Res}_R(Z\cup B)$. Since $\deg (R\cap Z)=1$, we have $\deg (\mbox{Res}_R(Z)) =3$. Since $R'$ contains the degree $3$ subscheme  $\mbox{Res}_R(Z)$ of $Z$, we have $R'=L$. Since $B\cap L =\emptyset$, we get $B\subset R$, a contradiction.

Now assume the existence of a conic $C\subset\Ps S_1$ such that $\deg (C\cap (Z\cup B)) \ge 10$ (we do not assume that the conic is smooth). Since $\deg (\mbox{Res}_C(Z\cup B)) \le 1$, we have $h^1(\mathcal {I}_{\mbox{Res}_C(Z\cup B)}(2)) =0$. As in the proof of \cite[Theorem 1]{BB2}, or as in \cite[Lemma 5.1]{BB3}, with the degree two divisor $C$ instead of a hyperplane, and since $Z\cap B =\emptyset$, we get $Z\cup B \subset C$. Since $C$ is a conic and $\deg (Z\cap L) =3$, $L$ must be a component of $C$, say $C = L\cup L'$ with $L'$ a line (we allow the case $L' = L$). Since $L\subset X$, we have $B\cap L = \emptyset$. Hence $B\subset L'$. Since $\deg (C\cap (Z\cup B)) \ge 10$, we have $a\ge 6$. Hence $h^1(\mathcal {I}_B(4))>0$, a contradiction.
\end{example}

A more explicit example with a simpler check (due to an anonymous referee) can be given as follows.

\begin{example}\label{SE}
Let $S_\bullet=\mathbb{K}\left[x_0,x_1,x_2\right]$ and $f={x_0}^4+{x_1}^4+\left(x_0+x_1\right)^4+{x_2}^4$. For each $g\in S_\bullet$ let $\al(g)$ be the dimension of the space $\{\partial_tg:t\in S^\bullet\}$ (in other words, the length of the apolar algebra) and let $X\subset\Ps S_1$ be the line $x^2=0$, with $\Span{x^2}:=\Span{x_0,x_1}^\perp$. We have
\[
\Ork(f,X)\ge\al(f)-\al\left(\partial_{x^2}f\right)\;.
\]
(see \cite[Prop.~3]{BuT} which, as explicitly mentioned right before its statement, is a result that was essentially observed in \cite{DT}).  A calculation on each degree up to $4$ gives $\al(f)=1+3+4+3+1=12$ and $\al\left(\partial_{x^2}f\right)=1+1+1+1=4$. Hence $\Ork(f,X)\ge 8$. Note also that $f$ essentially depends on three variables, because $\partial_{x^0}f$, $\partial_{x^1}f$ and $\partial_{x^2}f$ are linearly independent.
\end{example}

We shall see later (see Remark~\ref{EO}) that $f$ essentially depends on three variables in the more general situation of Example~\ref{Esempio}. Therefore, the maximum open rank for essentially ternary quartics is at least~$8$. Note that the maximum rank for ternary quartics is $7$ instead (see \cite[Chap.~3]{K} or \cite{D}). Independently of this remark, even before Blekherman and Teitler's work \cite{BT} there were signs that the upper bounds given by the induction procedure on open rank are quite relaxed. On the other hand, the ability of giving nontrivial upper bound, in a relatively simple way, give another indication on the attention that the notion of open rank may deserve. By this reason, we start now proving that the maximum open rank of essentially ternary quartics is actually $8$ (in particular, it is strictly less than the value given by the induction procedure based on maximum open rank in degree $3$).

\section{Maximum Open Rank for Ternary Quartics}\label{Tq}

Our goal in this section is to prove that $\Ork f\le 8$ for all $f\in S_4$ when $\dim S_1=3$. We need some auxiliary results, some of them of independent interest.

\begin{proposition}\label{Duerette}
Assume $\dim S_1=3$ and that a closed $X\subsetneq\Ps S_1$ is given. Let $f\in S_4$ and suppose that there exist distinct $\Span{l^0},\Span{l^1}\in\Ps S^1$ such that 
\begin{itemize}
\item none of the lines $l^0=0$, $l^1=0$ in $\Ps S_1$ is contained in $X$,
\item $\partial_{l^0l^1}f=0$.
\end{itemize}
Then $\Ork(f,X)\le 8$.
\end{proposition}
\begin{proof}
For each $i\in\{0,1\}$, let $X_i$ be the intersection of $X$ with the line $l^i=0$ and let us consider the dually paired graded rings $R_i^\bullet:=S^\bullet/\left(l^i\right)$ and $R_{i,\bullet}:=\Ker\partial_{l^i}\subset S_\bullet$. Let us pick $f_0\in R_{0,4}$ such that $\partial_{l^1+\left(l^0\right)}f_0=\partial_{l^1}f$ and set $f_1:=f-f_0$. We have $f=f_0+f_1$ with $f_0\in R_{0,4}$, $f_1\in R_{1,4}$.

Suppose first that $\partial_{l^0}f\ne0$ and $\partial_{l^1}f\ne 0$. Let $\Span{v_{01}}:=\Span{l^0,l^1}^\perp$ and note that $f_0,f_1$ can be replaced with $f_0+\lambda{v_{01}}^4$, $f_1-\lambda{v_{01}}^4$, for any $\lambda\in\mathbb{K}$. Moreover, each of $f_0-\lambda{v_{01}}^4$ and $f_1-\lambda{v_{01}}^4$ may be a fourth power of a linear form for at most two values of $\lambda$ (see, e.g., \cite[Rem.~2.2]{D2} and take into account that $f_0,f_1\not\in\Span{{v_{01}}^4}$ because of the assumption $\partial_{l^0}f\ne0$, $\partial_{l^1}f\ne 0$). Hence we can assume that $f_0,f_1$ are not fourth powers. Since they can be regarded as binary forms, by Proposition~\ref{SOr} we have that their open ranks as such, which we denote by $\Ork_{R_0}f_0$ and $\Ork_{R_1}f_1$, are at most $4$. Therefore \[\Ork(f,X)\le\Ork\left(f_0,X\right)+\Ork\left(f_1,X\right)\le\Ork_{R_0}\left(f_0,X_0\right)+\Ork_{R_1}\left(f_1,X_1\right)\le 8\;.\]

When $\partial_{l^0}f=0$ or $\partial_{l^1}f=0$, $f$ can be regarded as a binary form and we deduce $\Ork(f,X)\le 5< 8$ from Proposition~\ref{SOr}.\hfill $\square$
\end{proof}

\begin{lemma}\label{Distinti}
Assume $\dim S_1=3$, let $g\in S_3$ and $\Sigma\subset\Ps S^1$ be a finite set such that
\[\partial_{l'l''}g\ne 0\;,\qquad\forall \Span{l'},\Span{l''}\in\Sigma\;.\]
Then there exist distinct $\Span{l},\Span{m}\in\Ps S^1\setminus\Sigma$ such that
\[
\partial_{lm}g=0\;.
\]
\end{lemma}
\begin{proof}
The dimension of $L:=\Ker g_{2,1}$ is at least three because the partial polarization $g_{2,1}$ maps $S^2$ into $S_1$. Since the locus $X\subset\Ps S^2$ given by reducible forms is a hypersurface, we have that the intersection $Y:=\Ps L\cap X$ is an algebraic set of dimension at least one. For distinct $\Span{a}\,,\Span{b}\in\Ps S^1$, we have that if $\Span{a^2},\Span{b^2}\in Y$, $\lambda\in\mathbb{K}$, then $\Span{a^2+\lambda b^2}\in Y$, and $a^2+\lambda b^2$ is a simply degenerate quadratic form for all $\lambda\ne 0$ ($\operatorname{char}\mathbb{K}=0\ne 2$). We deduce that the set of all $\Span{lm}\in Y$ with distinct $\Span{l},\Span{m}\in\Ps S^1$, is a dense open subset $U\subseteq Y$. We have to choose $\Span{lm}\in U$ with $\Span{l},\Span{m}\not\in\Sigma$.

We can certainly assume that there exist $\Span{r}\in\Sigma$ and two distinct points $\Span{x},\Span{y}\in\Ps S^1$ such that $\Span{rx},\Span{ry}\in Y$, otherwise the required $\Span{l},\Span{m}$ can obviously be found, since $U$ is an infinite set. Let us fix such $\Span{r}, \Span{x}, \Span{y}$. They are linearly independent because $r\in\Span{x,y}$ would lead to $\partial_{r^2}g=0$, contrary to the hypothesis on $\Sigma$. Hence we have dually paired graded rings $R^\bullet:=\mathbb{K}[x,y]\subset S^\bullet$ and $R_\bullet:=\Ker\partial_{r}\subset S_\bullet$. Since $\partial_{rx}g=0$ and $\partial_{ry}g=0$ we have $g=v^3+h$ for some $h\in R_3$ and $v\in\Span{x,y}^\perp\subset S_1$. Since the partial polarization $h_{2,1}$ maps $R^2$ into $R_1$, and $R^\bullet$ is a ring of binary forms, we can find nonzero $l,t\in R^1$ such that $\partial_{lt}h=0$. By the hypothesis on $\Sigma$ we can assume that $\Span{l}\not\in\Sigma$. For infinitely many $\lambda\in\mathbb{K}$ we have $\Span{t+\lambda r}\not\in\Sigma\cup\left\{\Span{l}\right\}$, and let us fix $m:=t+\lambda r$ for whatever one of them. Since $\partial_{l}v=0$, $\partial_{r}h=0$, $\partial_{lt}h=0$ and $g=v^3+h$, we conclude that $\partial_{lm}g=0$.\hfill $\square$
\end{proof}

\begin{proposition}\label{Rette}
Assume $\dim S_1=3$, let $f\in S_d$ with $d\ge 3$ and suppose that $\Sigma\subset\Ps S^1$ is a finite set such that $\partial_{l'l''}f\ne 0$ for all $\Span{l'},\Span{l''}\in\Sigma$. Then there exist distinct \[\Span{l^1},\ldots,\Span{l^{d-1}}\in\Ps S^1\setminus\Sigma\] such that $\partial_{l^1\cdots l^{d-1}}f=0$.
\end{proposition}
\begin{proof}
The case $f=0$ being trivial, let us assume $f\ne 0$. Recall that for every nonzero $h\in S_e$ and nonzero $x\in S^1$, we have that $\partial_{x^m}h=0$ if and only if $\Span{x}$ is a point of multiplicity at least $e+1-m$ of the curve $h=0$ in $\Ps S^1$. In particular, when $m\le e$, the set of all $\Span{x}\in\Ps S^1$ with $\partial_{x^m}h\ne 0$ is nonempty and open. Exploiting this simple fact, we can inductively pick distinct $\Span{l^3},\ldots,\Span{l^{d-1}}\in\Ps S^1\setminus\Sigma$ such that \[\partial_{l^il^jl^3\cdots l^{d-1}}f\ne 0\;,\qquad\forall\Span{l^i},\Span{l^j}\in\Sigma':=\Sigma\cup\left\{\;\Span{l^3}\,,\,\ldots\,,\,\Span{l^{d-1}}\;\right\}\;.\]
Then the result follows from Lemma~\ref{Distinti} with $g:=\partial_{l^3\cdots l^{d-1}}f$ and $\Sigma'$ in place of~$\Sigma$.\hfill $\square$
\end{proof}

Exploiting the above proposition in the case when $f$ is a quartic, we can keep three distinct lines, $l^1=0$, $l^2=0$, $l^3=0$ with $\partial_{l^0l^1l^2}f=0$, from falling into a given forbidden locus $X$, unless $\partial_{l'l''}f=0$ for some lines $l'=0$, $l''=0$ (which unfortunately fall into $X$ and may be not distinct). A decomposition procedure along three lines, similar to that along two which was used in the proof of Proposition~\ref{Duerette}, looks promising. This idea has been successful for the maximum rank of ternary quartics: see \cite[Prop.~3.1 and~5.1]{D}. In that case as well, a condition $\partial_{l'l''}f=0$ with coinciding lines needed to be worked out separately (\cite[Prop.~5.2]{D}). One of the outcomes of the present work is that the strategy used for \cite[Prop.~3.1]{D} can be considerably simplified under the hypothesis that $\partial_{l^1l^2}f$ is not a square. Note that this condition is only slightly stronger: we wish that the rank of $\partial_{l^1l^2}f$ is at least two, whereas $\partial_{l'l''}f\ne 0$ means that the rank is at least one.

A further simplification of the line of proof of \cite[Prop.~3.1]{D} was set up in \cite{D3}. To let it work with open rank, we need to slightly adapt \cite[Lemmas~2.6 and 2.7]{D3}, by adding the information that the decompositions given in those lemmas can be chosen out of a given algebraic set (we shall also slightly weaken the hypothesis of \cite[Lemma~2.7]{D3}). Let us recall that when $\dim S_1=3$, $f\in S_d$ and $\partial_lf=0$ for some nonzero $l\in S^1$, $f$ can be regarded as a binary form in the graded subring $R_\bullet:=\Ker\partial_l\subset S$, which is dually paired with the quotient $S/\left(l\right)$ in a natural way. In this case, the length of $f$ as an element of $R_d$ does not depend on the choice of $l$. We call it the \emph{binary length} of $f\in S_d$ and denote it by $\blen f$ (see \cite[Def.~2.1]{D3} for a general definition).

\begin{lemma}\label{D327}
Let $\Span{g'}\in\Ps S_d$ with $\dim S_1=3$, $d>0$, and let us write $d=2s+\varepsilon$, with $\varepsilon\in \{0,1\}$ and $s$ integer. Let $\Span{l^0},\ldots,\Span{l^t}\in\Ps S^1$ be distinct and such that $\partial_{l^0}g'=0$, and for each $i\in\{1,\ldots, t\}$ let $g_i\in S_{d+1}$ be such that $\partial_{l^i}g_i=g'$. Moreover, suppose that a closed subset $Y\subset\Ps S_1$ that does not contain the line $l^0=0$ is given. If
\[
\blen g'=s+1\;,\quad\blen\partial_{l^0}g_1\ge s+\varepsilon\;,\quad\ldots\quad,\blen\partial_{l^0}g_t\ge s+\varepsilon\;,
\]
then there exists a power sum decomposition
\begin{equation}\label{Decv}
g'={v_1}^d+\cdots+{v_r}^d
\end{equation}
such that: $r\le s+1+\varepsilon$, $Y$ contains none of $\Span{v_1},\ldots,\Span{v_r}$ and, for each $i\in\{1,\ldots, t\}$,
\begin{itemize}
\item $l^i$ vanishes on none of $v_1,\ldots, v_r$,
\item $\blen\left(g_i-F_i\right)=s+1+\varepsilon$, where
\[
F_i:=\frac1{(d+1)l^i\left(v_1\right)}{v_1}^{d+1}+\cdots+\frac1{(d+1)l^i\left(v_r\right)}{v_r}^{d+1}\;.
\]
\end{itemize}
\end{lemma}
\begin{proof}
The proof can go in the same way as that of \cite[Lemma~2.7]{D3}, with the following additional cautions.

At the beginning of that proof, dually paired rings $R_0^\bullet:=S^\bullet/\left(l^i\right)$ and $R_{0,\bullet}:=\Ker\partial_{l^i}\subset S_\bullet$ are considered (among others $R_i$s). Then a line $\Ps L$ in a subspace $\Ps H\le\Ps R_0^{s+1+\varepsilon}$ is chosen. The vectors $v_1,\ldots, v_r$ are the roots in $\Ps R_{0,1}$ (the line $l^0=0$) of a form $h$, with $\Span{h}$ chosen in a suitable cofinite subset of $\Ps L$, say it~$V$.

The above choices are allowed by \cite[Lemma~2.6]{D3}. In the proof of that lemma, two coprime generator of the apolar ideal of $g'$ are used (and denoted by $l$ and $h^0$). The fact that they are coprime easily implies that the algebraic set $\widetilde{Y}$ of all $\Span{h}\in\Ps H$ that have at least one root in $Y$ does not fill $\Ps H$. Hence, in the proof of \cite[Lemma~2.7]{D3}, the line $\Ps L$ can certainly be chosen with the additional property of being not contained in $\widetilde{Y}$. Therefore $\Span{h}$ can be chosen in $V\setminus\widetilde{Y}$, because that set is cofinite in $\Ps L$ as well.

Moreover, note that a condition $\blen\partial_{l^0}g_1=\cdots=\blen\partial_{l^0}g_t=s+1$ is used in the proof of \cite[Lemma~2.7]{D3}, but only to get \cite[Eq.~(8)]{D3}. It is easy to see that \cite[Eq.~(8)]{D3} holds also under the weaker hypothesis $\blen\partial_{l^0}g_1\ge s+\varepsilon,\ldots,\blen\partial_{l^0}g_t\ge s+\varepsilon$, if one takes into account that those binary lengths can not exceed $s+1$, because $d=2s+\varepsilon$.\hfill $\square$
\end{proof}

Now the proof of \cite[Prop.~3.2]{D3} can easily be adapted as follows.

\begin{proposition}\label{Decomp}
Assume $\dim S_1=3$ and that a closed $X\subsetneq\Ps S_1$ is given. Let $f\in S_4$ and suppose that there exist distinct $\Span{l^0},\Span{l^1},\Span{l^2}\in\Ps S^1$ such that 
\begin{itemize}
\item none of the lines $l^0=0$, $l^1=0$, $l^2=0$ in $\Ps S_1$ is contained in $X$,
\item $\partial_{l^0l^1l^2}f=0$,
\item $\partial_{l^1l^2}f$ is not a square.
\end{itemize}
Then $\Ork(f,X)\le 8$.
\end{proposition}
\begin{proof}
According to Proposition~\ref{Duerette}, we can assume $\partial_{l^0l^1}f\ne 0$, $\partial_{l^0l^2}f\ne0$. Hence we can exploit Lemma~\ref{D327} with $t=2$, $g'=\partial_{l^1l^2}f$, $g_1=\partial_{l^2}f$, $g_2=\partial_{l^1}f$ and $Y=X$. We get $\Span{v_1},\Span{v_2}\in\Ps S_1\setminus X$ such that
\[
\partial_{l^1l^2}f={v_1}^2+{v_2}^2\;,
\]
$l^i\left(v_j\right)\ne 0$ for all $i,j\in\{1,2\}$, and setting
\[
F'_1:=\frac1{3l^1\left(v_1\right)}{v_1}^3+\frac1{3l^1\left(v_2\right)}{v_2}^3\;,\quad F'_2:=\frac1{3l^2\left(v_1\right)}{v_1}^3+\frac1{3l^2\left(v_2\right)}{v_2}^3\;,
\]
we have $\blen\left(\partial_{l^2}f-F'_1\right)=\blen\left(\partial_{l^1}f-F'_2\right)=2$.

Let
\[
F_1:=\frac1{12l^1\left(v_1\right)l^2\left(v_1\right)}{v_1}^4+\frac1{12l^1\left(v_2\right)l^2\left(v_2\right)}{v_2}^4
\]
and let us exploit again Lemma~\ref{D327}, now with $t=1$, $l^1, l^2$ in place of $l^0, l^1$, $\partial_{l^2}f-F'_1$ in place of $g'$, $f-F_1$ in place of $g_1$ and $Y=X$. We get \[\Span{w_1},\ldots,\Span{w_r}\in\Ps S_1\setminus X\] such that $r\le 3$,
\[
\partial_{l^2}f-F'_1={w_1}^3+\cdots+{w_r}^3\;,
\]
$l^2\left(w_i\right)\ne 0$ for all $i\in\{1,\ldots ,r\}$, and setting
\[
G_2:=\frac1{4l^2\left(w_1\right)}{w_1}^4+\cdots+\frac1{4l^2\left(w_r\right)}{w_r}^4
\]
we have $\blen\left(f-F_1-G_2\right)=3$. We have $\Ork_{R_2}\left(f-F_1-G_2\right)=3$ by Proposition~\ref{SOr}. Since $X$ does not contain the line $l^2=0$ we also have \[\Ork\left(f-F_1-G_2,X\right)\le\Ork_{R_2}\left(f-F_1-G_2\right)=3\;.\] We conclude that
\[
\Ork(f,X)\le\Ork\left(f-F_1-G_2,X\right)+r+2\le3+3+2=8\;.
\]
\hfill $\square$
\end{proof}

Our next job is to work out the special case where $\partial_{l^1l^2}f$ is a square. Of course, since the indices can be rearranged, we can suppose that $\partial_{l^1l^2}f$, $\partial_{l^0l^2}f$ and $\partial_{l^0l^1}f$ are all squares.

\begin{lemma}\label{Square}
Assume $\dim S^1=3$ and that a finite set $\Sigma\subset\Ps S^1$ is given. Let $f\in S_4$ and suppose that, whenever $\Span{l^0},\Span{l^1},\Span{l^2}\in\Ps S^1\setminus\Sigma$ are distinct and such that $\partial_{l^0l^1l^2}f=0$, we have that $\partial_{l^1l^2}f$, $\partial_{l^0l^2}f$ and $\partial_{l^0l^1}f$ are all squares.

Then there exist nonzero $x,y\in S^1$ such that $\partial_{xy}f=0$. Moreover, for each fixed $l^0,l^1,l^2$ as before (when they exist), we can take $x=l^1$.
\end{lemma}
\begin{proof}
Assume first that distinct \[\Span{l^0},\Span{l^1},\Span{l^2}\in\Ps S^1\setminus\Sigma\;,\] such that $\partial_{l^0l^1l^2}f=0$, do exist. Then $\partial_{l^1l^2}f=v^2$ for some $v\in S_1$. Therefore $\partial_{l'l^1l^2}f=0$ for all $l'\in\Span{v}^\perp$. Hence, for all $\Span{l'}\in\Ps\Span{v}^\perp\setminus\left(\Sigma\cup\left\{\Span{l^1},\Span{l^2}\right\}\right)$, $\partial_{l^1l'}f$ is a square. By \cite[Lemma~4.1]{D2} (for a correct statement of that lemma, $f\in S_d$ must be replaced with $f\in S_{d+1}$, $d\ge 2$), we can find a nonzero $l'\in\Span{v}^\perp$ such that $\partial_{l^1l'}f=0$. Therefore, it suffices to set $x=l^1,y=l'$.

In the case when it is not possible to find distinct $\Span{l^0},\Span{l^1},\Span{l^2}\in\Ps S^1\setminus\Sigma$ with $\partial_{l^0l^1l^2}f=0$, Proposition~\ref{Rette} assures that $\partial_{xy}f=0$ for some $\Span{x},\Span{y}\in\Sigma$.\hfill $\square$
\end{proof}

Now the case where the simplifying assumption in Proposition~\ref{Decomp} is missed has been reduced to the case where $\partial_{xy}f=0$ for some $\Span{x},\Span{y}\in\Ps S^1$. When $\Span{x}=\Span{y}$ that condition becomes the same as in the hypothesis of \cite[Prop.~5.2]{D}. The basic idea in the proof of that proposition can be illustrated for $\Span{x}\ne\Span{y}$ as follows.

\begin{lemma}\label{Start}
Assume $\dim S_1=3$ and let $f\in S_4$. Suppose that:
\begin{itemize}
\item there are distinct $\Span{x},\Span{y}\in\Ps S^1$ with $\partial_{xy}f=0$,
\item there is $\Span{w}\in\Ps S^3$ with $\partial_wf=0$ and
\item the curve $w=0$ intersects the lines $x=0$ and $y=0$ in two groups of distinct points $P_0,P_1,P_2$ and $Q_0,Q_1,Q_2$, not coinciding with the intersection point $O=\Span{x,y}^\perp$.
\end{itemize}
Then, if $l^i=0$ is the line through $P_i$ and $Q_i$ for each $i$ (with $l^i\in S^1$), we have $\partial_{l^0l^1l^2}f=0$.
\end{lemma}
\begin{proof}
The curve $l^0l^1l^2=0$ contains the complete intersection of $w=0$ and $xy=0$ (a set of six distinct points, different from $O$). Then $l^0l^1l^2=\lambda w+mxy$ for some $\lambda\in\mathbb{K}$ and $m\in S^1$, by elementary intersection theory in algebraic geometry. Since $\partial_wf=0$ and $\partial_{xy}f=0$, we get $\partial_{l^0l^1l^2}f=\lambda\partial_wf+\partial_m\partial_{xy}f=0$.
\hfill $\square$
\end{proof}

The cubic $w$ in the above statement can actually be found, except for a special case. This fact will be stated in Lemma~\ref{Cut} below, along with an additional property of $w$ which implies that, for the above obtained three lines, $\partial_{l^1l^2}f$ is not a square. As a matter of facts, the outcome of the subsequent Lemmas~\ref{Cross} and~\ref{Sdoppia} is precisely that the simplifying assumption needed in the hypothesis of Proposition~\ref{Decomp} can not be missed, if not in a special case: when $\partial_zf$ is a cube for some nonzero $z\in S_1$ (equivalently, its rank is at most one). This result will be refined further by Lemma~\ref{Cc} and explicitly stated in Proposition~\ref{Predecomp}.

To begin with, we recall a fact already pointed out at the beginning of \cite[proof of Proposition~5.2]{D}.

\begin{lemma}\label{Dim}
Assume $\dim S_1=3$, let $f\in S_4$, $l\in S^1$ and suppose that $\partial_lf$ is not a cube. Then
\[
\dim\left(\Ker f_{3,1}\cap lS^2\right)\le 4\;.
\]
\end{lemma}
\begin{proof}
Let $g:=\partial_lf$. The space $\Ker f_{3,1}\cap lS^2$ is isomorphic to $\Ker g_{2,1}$ through multiplication by $l$. Then its dimension equals $6-\rk g_{2,1}$. A general and easy result is that if $p\in S_{d+\delta}$ for some positive integers $d,\delta$, then $\rk p_{\delta,d}=\rk p$ whenever one of these numbers is at most one (this holds regardless of $\dim S_1$). Hence $\dim\left(\Ker f_{3,1}\cap lS^2\right)\le 4$ if and only if $\rk g\ge 2$, that is, $\partial_lf$ is not a cube.\hfill $\square$
\end{proof}

With a bit of extra work we get the following.

\begin{lemma}\label{Dim3}
Assume $\dim S_1=3$, let $f\in S_4$, $\Span{x},\Span{y}\in\Ps S^1$ be distinct and suppose that $\partial_{xy}f=0$, $\partial_xf\ne 0$ and $\partial_yf$ is not a cube. Then
\[
\dim\left(\Ker f_{3,1}\cap yS^2\right)=4\quad\text{and}\quad\dim\Ker f_{3,1}=7\;.
\]
\end{lemma}
\begin{proof}
For every $w\in\Ker f_{3,1}+yS^2$ we have $\partial_{xw}f=0$, because $\partial_{xy}f=0$. Since $\partial_xf\ne 0$, there exists $w\in S^3$ such that $\partial_{xw}f\ne 0$, hence $w\not\in\Ker f_{3,1}+yS^2$. This shows that $\dim\left(\Ker f_{3,1}+yS^2\right)\le\dim S^3-1=9$, hence
\begin{multline*}
\dim\Ker f_{3,1}-\dim\left(\Ker f_{3,1}\cap yS^2\right)=\dim\left(\Ker f_{3,1}+yS^2\right)-\dim yS^2\\\le 9-6=3\;.
\end{multline*}
But $\dim\Ker f_{3,1}\ge\dim S^3-\dim S_1=7$ and, by the preceding lemma, we also have $\dim\left(\Ker f_{3,1}\cap yS^2\right)\le 4$, so the result readily follows.\hfill $\square$
\end{proof}

The above technical result has the following useful outcome.

\begin{lemma}\label{Cut}
Assume $\dim S_1=3$, let $f\in S_4$, $W:=\Ker f_{3,1}$, $\Span{x},\Span{y}\in\Ps S^1$ be distinct and $X$ be a finite subset of the line $y=0$. Suppose that $\partial_{xy}f=0$ and that $\partial_zf$ is not a cube for every nonzero $z\in S^1$. Then there exists a nonempty (Zariski) open subset $U\subset\Ps W$ such that for all $\Span{u}\in U$ the curve $u=0$ intersects $y=0$ in three distinct points outside $X$, and every form in $W$ that vanishes on two of them, vanishes on the other point too.
\end{lemma}
\begin{proof}
The ring $R^\bullet:=S^\bullet/(y)$ can be regarded as the graded ring of the line $y=0$, and for each $u\in S^d$ the intersection of $u=0$ and $y=0$ is the zero locus of $\overline u:=u+(y)\in R^d$. Then the linear system on $y=0$ cut by all $\Span{w}\in\Ps W$ is given by the groups of roots of the forms in $\overline W:=W/\left(W\cap yS^2\right)$. Since $\partial_zf$ is not a cube for every nonzero $z\in S^1$, the linear system of curves given by $W:=\Ker f_{3,1}$ has no base points by \cite[Lemma~2.1]{K}. Henceforth, the linear system on the line $y=0$ given by $\overline{W}$ has no base points, and by Lemma~\ref{Dim3} we have $\dim\overline{W}= 3$. Hence there is a nonempty open subset $\overline U\subset\Ps\overline W$ such that for every $\Span{\overline u}\in\overline U$ we have:
\begin{itemize}
\item$\overline{u}$ has three distinct roots outside $X$,
\item every form in $\overline{W}$ that vanishes on two of them, vanishes on the other root too.
\end{itemize}
Then it suffices to take $U$ as the preimage of $\overline U$ through the natural projection $\Ps W\setminus\Ps\left(yS^2\right)\to\Ps\overline W$.\hfill $\square$
\end{proof}

\begin{lemma}\label{Cross}
Assume $\dim S_1=3$ and that a finite set $\Sigma\subset\Ps S^1$ is given. Let $f\in S_4$ and suppose that there exist distinct $\Span{x},\Span{y}\in\Ps S^1$ such that $\partial_{xy}f=0$ and that $\partial_zf$ is not a cube for every nonzero $z\in S^1$. Then there exist distinct $\Span{l^0},\Span{l^1},\Span{l^2}\in\Ps S^1\setminus\Sigma$ such that $\partial_{l^0l^1l^2}f=0$ and $\partial_{l^1l^2}f$ is not a square.
\end{lemma}
\begin{proof}
Let $X$ be the set of all points on the line $y=0$ that belongs to $l^i=0$ for some $l^i\in\left(\Sigma\cup\left\{\Span{x}\right\}\right)\setminus\left\{\Span{y}\right\}$, and let $Y$ be similarly defined for the line $x=0$. We can exploit Lemma~\ref{Cut} for $x$, $y$, $X$, and also for $y$, $x$, $Y$ in place of them (respectively). We get nonempty open subsets of $\Ps\Ker f_{3,1}$, and whatever chosen $\Span{w}$ in their (nonempty) intersection fulfills the requirements in Lemma~\ref{Start}. That lemma gives three distinct lines $l^0=0$, $l^1=0$, $l^2=0$ with $\partial_{l^0l^1l^2}f=0$, and by construction we have $l^1,l^2,l^3\in\Ps S_1\setminus\Sigma$. To exclude that $\partial_{l^1l^2}f=v^2$ for some $v\in S_1$, note that in this case $l^0(v)=0$, hence $\Span{v}$ can not be the intersection point of $x=0$ and $y=0$. Therefore we can pick $\Span{m^0}\in\Span{v}^\perp\setminus\left\{\Span{l^0}\right\}$, and the two lines $m^0=0$, $l^0=0$ intersect at least one of the lines $x=0$, $y=0$, say the first one, in different points. But this is excluded because $\partial_{m^0l^1l^2}f=0$ and the curve $m^0l^1l^2=0$ shares with $w=0$ two intersections with $x=0$, but not the other.\hfill $\square$
\end{proof}

For the case where $\partial_{xy}=0$ with $\Span{x}=\Span{y}$, we can follow the line of the proof of \cite[Proposition~5.2]{D}.

\begin{lemma}\label{Sdoppia}
Assume $\dim S^1=3$ and that a finite set $\Sigma\subset\Ps S^1$ is given. Let $f\in S_4$ and suppose that there exists a nonzero $l'\in S^1$ such that $\partial_{{l'}^2}f=0$, and that $\partial_zf$ is not a cube for every nonzero $z\in S^1$. Then there exist distinct $\Span{l^0},\Span{l^1},\Span{l^2}\in\Ps S^1\setminus\Sigma$ such that $\partial_{l^0l^1l^2}f=0$ and $\partial_{l^1l^2}f$ is not a square.
\end{lemma}
\begin{proof}
Let $V:=\Ker f_{3,1}$, $W:=V\cap l'S^2$. By Lemma~\ref{Dim} we have  $\dim W\le 4$. In the second part of the proof of \cite[Proposition~5.2]{D} distinct points $P_0,P_1,P_2$ on the line $l'=0$ are chosen. In the present situation, we can furthermore impose that none of them belongs to $l''=0$ for any $\Span{l''}\in\Sigma\setminus\{\Span{l'}\}$, unless the linear system that is cut on the line $l'=0$ by the curves $p=0$ with $\Span{p}\in\Ps V$ admits a fixed point; but this is excluded by \cite[Lemma~2.1]{K}, because of our hypothesis that $\partial_zf$, with $\Span{z}\in\Ps S^1$, is never a cube. Then we can find, as in the mentioned proof, distinct $\Span{x^0},\Span{x^1},\Span{x^2}\in\Ps S^1$ such that $\partial_{x^0x^1x^2}f=0$ and for each $i$, $x^i=0$ meets $l'=0$ in $P_i$ only. It readily follows that $\Span{x^0},\Span{x^1},\Span{x^2}\in\Ps S^1\setminus\Sigma$. At this point we do not know if $\partial_{x^1x^2}f$ may be a square, but we can find the required $\Span{l^0},\Span{l^1},\Span{l^2}$ as follows.

Suppose that there exists a nonzero $y^1\in S^1$ such that $\partial_{x^1y^1}f=0$. If $\Span{x^1}\ne\Span{y^1}$ the statement follows from Lemma~\ref{Cross} with $x^1,y^1$ in place of $x,y$. When $\Span{x^1}=\Span{y^1}$ we have $\partial_{{x^1}^2}f=0$ and the statement follows again from Lemma~\ref{Cross}, now with $l'+x^1$, $l'-x^1$ in place of $x,y$. Therefore, we can assume that $\partial_{x^1y^1}f\ne 0$ for all nonzero $y^1\in S^1$. But with this assumption the statement follows from Lemma~\ref{Square}.\hfill $\square$
\end{proof}

Now we work out the special case where $\partial_zf$ is a cube for some nonzero $z\in S_1$, at the cost of leaving out an even more special case, which we now briefly introduce. As mentioned before Proposition~\ref{SOr}, in \cite[Def.~5.66]{IK} some extensions of the notion of length are presented. In particular, let us recall the notion of \emph{scheme length}, which nowadays is often called \emph{cactus rank}. Given $f\in S_d$, its cactus rank (or scheme length) is the minimum among the degrees of the zero dimensional schemes $Z$ in $\operatorname{Proj} S^\bullet$, such that the ideal $I(Z)$ in $S^\bullet$ is contained in the apolar ideal $I_f=\left\{x\in S^\bullet:\partial_xf=0\right\}$ of $f$. We shall denote it by $\crk f$.

We have also the following more geometric interpretation of $\crk f$ when $f\ne 0$. Let $\Span{x}\in\Ps S^d$, and $\nu_d:S_1\to S_d$, $\nu_d(v):=v^d$, be the Veronese map. Let $I(Z)$ and $I_f$ be the ideal of $Z$ and the apolar ideal of $f$. Then $x\in I(Z)$ if and only if the hyperplane $\Ps\Span{x}^\perp$ in $\Ps S_d$ contains $\nu_d(Z)$, and $x\in I_f$ if and only if the same hyperplane contains $\Span{f}$. For forms $y\in S^e$ of lower degree, we have that $y\in I(Z)$ if and only if $yS^{d-e}\subseteq I(Z)$ and $y\in I_f$ if and only if $yS^{d-e}\subseteq I_f$. It easily follows that $I(Z)\subseteq I_f$ if and only if $\Span{f}\in\Span{\nu(Z)}$ (scheme-theoretic, projective span, under the natural identification of $\Ps S_1$ with the set of all closed points of $\operatorname{Proj}S^\bullet$ and the similar identification for $\Ps S_d$ and $\operatorname{Proj}\Sy^\bullet S^d$). Hence $\crk f$ is the minimum among the degrees of the zero dimensional schemes $Z$ such that $\Span{f}\in\Span{\nu(Z)}$.

\begin{lemma}\label{Cc}
Assume $\dim S_1=3$ and that a finite set $\Sigma\subset\Ps S^1$ is given. Let $f\in S_4$ with $\crk f\ge 4$, and suppose that $\partial_zf$ is a cube for some nonzero $z\in S^1$. Then there exist distinct $\Span{l^0},\Span{l^1},\Span{l^2}\in\Ps S^1\setminus\Sigma$ such that $\partial_{l^0l^1l^2}f=0$ and $\partial_{l^1l^2}f$ is not a square.
\end{lemma}
\begin{proof}
Let $\partial_zf=v^3$ for some $v\in S_1$ and nonzero $z\in S^1$. Let us consider the graded rings $R_\bullet:=\Ker\partial_z\subset S_\bullet$, $R^\bullet:=S^\bullet/(z)$, with the induced apolarity pairing. We shall distinguish the two cases $z(v)\ne 0$ and $z(v)=0$.

Suppose first that $z(v)\ne 0$. Let
\[
g:=f-\frac1{4z(v)}v^4\;,
\]
so that $\partial_zg=0$. We have $\blen g\ge 3$ because $\crk f\ge 4$ (if a subscheme $Z$ works for $g$ then $Z\cup\left\{\Span{v}\right\}$ works for $f$). Hence $\blen g=3$ because $\deg g=4$. By \cite[Theorem 1.44]{IK}, the apolar ideal $I_g\subset R^\bullet$ is generated by two coprime forms in $R^3$. Then we can find $\overline{l^0}\overline{l^1}\overline{l^2}\in I_g$, with $\overline{l^0},\overline{l^1},\overline{l^2}\in R^1$, such that its roots in $\Ps R_1$ are distinct and lie on no line $l'=0$ with $\Span{l'}\in\Sigma\setminus\left\{\Span{z}\right\}$. Note that $\partial_{\overline{l^1}\overline{l^2}}g\ne 0$ because $I_f$ is generated by degree $3$ forms, and is a square ${v_0}^2$, with $l^0\left(v_0\right)=0$ because $\partial_{\overline{l^0}\overline{l^1}\overline{l^2}}g=0$. We can certainly choose representatives $l^0,l^1,l^2\in S^1$ (that is, $\overline{l^i}=l^i+(z)\in R^\bullet=S^\bullet/(z)$) such that $l^0(v)=0$, $l^1(v)\ne 0$, $l^2(v)\ne 0$. We have $\partial_{l^0l^1l^2}g=0$ and $\partial_{l^0}v^4=0$, hence $\partial_{l^0l^1l^2}f=0$. We have that $\Span{l^0},\Span{l^1},\Span{l^2}$ are distinct and do not lie in $\Sigma$, because the lines $l^0=0$, $l^1=0$ and $l^2=0$ intersect $z=0$ in the roots of $\overline{l^0},\overline{l^1},\overline{l^2}$. Moreover, we have
\[
\partial_{l^1l^2}f={v_0}^2+\frac{3l^1(v)l^2(v)}{z(v)}v^2\;,
\]
which is not a square because $\Span{v_0}\ne\Span{v}$.

Suppose now $z(v)=0$. If $v=0$, then $f\in R_\bullet$ and this is excluded because $\blen f\le 3$ is incompatible with the hypothesis $\crk f\ge 4$. Let us pick $\Span{x}\in\Span{v}^\perp\setminus\left(\Sigma\cup\left\{\Span{z}\right\}\right)$ and set $h:=\partial_xf$, so that $\partial_zh=\partial_x{v}^3=0$. Since $\deg h=3$, we can find a nonzero $\overline{l^1}\overline{l^2}\in R^2$, with $\overline{l^1},\overline{l^2}\in R^1$ and $\partial_{\overline{l^1}\overline{l^2}}h=0$. If $\Span{\overline{l^1}},\Span{\overline{l^2}}$ can be chosen different from $\Span{\overline x}$, with $\overline x=x+(z)\in R^1$, then for whatever chosen representatives $l^1,l^2\in S^1$ we have $\partial_{xl^1l^2}f=0$ and since $l^1,l^2\not\in\Span{x,z}$ also $l^1(v)\ne 0$, $l^2(v)\ne 0$, so that $\partial_{zl^1l^2}f=\partial_{l^1l^2}v^3$ is a nonzero multiple of~$v$. If $\partial_{l^1l^2}f$ were a square $w^2$, we would have either $\partial_{zl^1l^2}f=0$ (in the case $w\in\Span{v}$) or that $\partial_{zl^1l^2}f$ is a nonzero multiple of $w$ with $\Span{w}\ne\Span{v}$. This show that $\partial_{l^1l^2}f$ is not a square, and hence it suffices to set $l^0:=x$ and take care of choosing $\Span{l^1},\Span{l^2}$ outside $\Sigma$. It remains to exclude that for every $\Span{\overline{l^1}},\Span{\overline{l^2}}\in\Ps R^1$ with $\partial_{\overline{l^1}\overline{l^2}}h=0$, at least one of them coincides with $\Span{\overline{x}}$. We shall more generally exclude, for whatever $\Span{\overline{l'}}\in\Ps R^1$, that $\partial_{\overline{x}\overline{l'}}h=0$.

Let us suppose the contrary and let $l'\in S_1$ be a representative of $\overline{l'}$. We have
\[
\partial_{xl'}f=\lambda v^2
\]
for some scalar $\lambda$, because $\partial_x\left(\partial_{xl'}f\right)=\partial_{xl'}h=0$ and $\partial_z\left(\partial_{xl'}f\right)=\partial_{l'}\partial_{xz}f=0$. It must be $\lambda\ne 0$, for otherwise the apolar ideal of $f$ would contain the ideal $\left(z^2,zx,xl'\right)$ of a degree three zero-dimensional scheme, in contrast with the hypothesis $\crk f\ge 4$. We also have
\[
\partial_{zl'}f=\partial_{l'}v^3=\mu v^2
\]
for some scalar $\mu$. Hence, setting $l'':=\lambda z-\mu x$, we have
\[
\partial_{l'l''}f=0\;.
\]
Then the apolar ideal of $f$ in $S^\bullet$ contains the ideal $I:=\left(z^2,zx,l'l''\right)$. If $\Span{l''}\ne\Span{z}$ (i.e., $\mu\ne 0$), then $I$ is the ideal of a degree three zero-dimensional scheme, which is excluded because $\crk f\ge 4$. If $\Span{l''}=\Span{z}$, we have $\partial_{l'}v^3=\partial_{l'z}f=0$, hence $\Span{\overline{l'}}=\Span{\overline{x}}$. In this case we have $\partial_{x^3}f=0$, hence the apolar ideal of $f$ contains $I':=\left(z^2,zx,x^3\right)$. Then \cite[Lemma~2.3]{BB} predicts that $\Span{f}$ is in the span of the image through the Veronese map of some curvilinear zero-dimensional scheme supported on $\Span{v}$ and of degree less than $4$; this is excluded because $\crk f\ge 4$. This ends the proof, but we also mention that such a scheme may also be explicitly exhibited. To this end, let us choose $y\in\Ps S_1\setminus\Span{x,z}$ and consider the ideal $\left(z^2,zx,\lambda'zy-\mu'x^2\right)$, where $\lambda',\mu'$ are the scalars determined by the relations $\partial_{x^2}f=\lambda'v^2$, $\partial_{zy}f=\mu'v^2(\ne 0)$. When $\lambda'\ne 0$, this is the ideal of the mentioned curvilinear scheme. When $\lambda'=0$, the scheme is of degree three (which suffices for the purposes of the proof) but not curvilinear; one may detect in a similar way an apolar to $f$ linear generator of the ideal of a smaller scheme.\hfill $\square$
\end{proof}

Let us summarize the information given by the above results in the following proposition.

\begin{proposition}\label{Predecomp}
Assume $\dim S^1=3$, let $f\in S_4$ with $\crk f\ge 4$, and $\Sigma$ be a finite subset of $\Ps S^1$. Then there exist distinct $\Span{l^0},\Span{l^1},\Span{l^2}\in\Ps S^1\setminus\Sigma$ such that $\partial_{l^0l^1l^2}f=0$ and $\partial_{l^1l^2}f$ is not a square.
\end{proposition}
\begin{proof}
If $\partial_{xy}f\ne 0$ for all nonzero $x,y\in S^1$, the result follows from Proposition~\ref{Rette} and Lemma~\ref{Square}. If $\partial_zf$ is a cube for some nonzero $z\in S_1$, the result follows from Lemma~\ref{Cc}. When $\partial_{xy}=0$ for some nonzero $x,y\in S^1$ and $\partial_zf$ is not a cube for every nonzero $z\in S_1$, the result follows from Lemma~\ref{Cross} if $\Span{x}\ne\Span{y}$, and from Lemma~\ref{Sdoppia} (with $l':=x$) if $\Span{x}=\Span{y}$.\hfill $\square$
\end{proof}

Now we work out the special case $\crk f=3$, using similar techniques as in Example~\ref{Esempio}.

\begin{proposition}\label{Br3}
Let $f\in S_4$, with $\dim S_1=3$. If $\crk f=3$, then $\Ork f=7$.
\end{proposition}
\begin{proof}
Let $\nu:\Ps S_1\to\Ps S_4$ be the Veronese embedding and set $P:=\Span{f}\in\Ps S_4$. If $\crk f=3$, then there exists a degree three zero-dimensional subscheme $Z$ of $\Ps S_1$, such that $P$ is in the $\Span{\nu(Z)}$, and $P\not\in\Span{\nu\left(Z'\right)}$ for every subscheme $Z'\subsetneq Z$. By \cite[Lemma~2.3]{BB}, $Z$ is curvilinear.

We first check that $\Ork f \ge 7$. Assume that $a:=\Ork f\le 6$. Fix a closed subset $X\subsetneq\Ps S_1$. Since $Z$ is curvilinear, there are only finitely many lines $L\subset\Ps S_1$ such that $\deg (L\cap Z) \ge 2$. Increasing if necessary $X$, we may assume that $X$ contains the union of these lines. In particular, $X$ contains $Z_{red}$. Take a degree $a$ reduced subscheme $B$ of $\Ps S_1\setminus X$, with $P\in\Span{\nu(B)}$ and such that $P\notin\langle\nu\left(B'\right)\rangle$ for any $B'\subsetneq B$. Since $Z_{red}\subset X$ and $B\subset\Ps S_1\setminus X$, we have $B\cap Z =\emptyset$, and in particular $B\ne Z$. But $P\in\Span{\nu(B)}\cap\Span{\nu(Z)}$, hence $h^1(\mathcal {I}_{Z\cup B}(4))>0$. Since $\deg (Z\cup B)=3+a\le 9$, by \cite[Lemma 34]{BGI} there is a line $L\subset \mathbb{P}^2$ such that $\deg (L\cap (Z\cup B))\ge 6$. For any effective divisor $D\subset \mathbb {P}^2$ and any zero-dimensional scheme $W\subset\Ps S_1$ let $\mbox{Res}_D(W)$ denote the residual scheme of $W$ with respect to $D$. We have $\deg (W) =\deg (D\cap W) + \deg (\mbox{Res}_D(W))$. Since $\deg (L\cap (Z\cup B)) \ge 6$, we have $\deg (\mbox{Res}_L(Z\cup B)) \le 3$. Since $Z\cap B =\emptyset$, by \cite[Lemma 5.1]{BB3} we get $Z\cup B \subset L$. Since $X$ contains any line $L$ with $\deg (L\cap Z)\ge 2$ and $B\cap X =\emptyset$, we get a contradiction.

Now we check that $\Ork f \le 7$. Fix a closed set $X\subsetneq\Ps S_1$ and let $C\subset\Ps S_1$ be a general conic containing $Z$. Since
$C$ is general, $C\nsubseteq X$. Since $Z$ is curvilinear and not contained in a line (otherwise $P\in\Span{\nu\left(Z'\right)}$ for some $Z'\subsetneq Z$), $C$ is a smooth conic. It follows that $P$ has border rank $3$ with respect to the rational normal curve $\nu(C)$. By Remark~\ref{SOrg}, the open rank of $P$ with respect to $\nu(C)$ is $7$, and therefore there exists $E\subset C\setminus C\cap X$ such that $\sharp(E)=7$ and $P\in\langle\nu(E)\rangle$. Hence $\Ork f\le 7$.\hfill $\square$
\end{proof}

\begin{remark}\label{EO}\rm
If $f$ is as in Example~\ref{Esempio}, then $\partial_lf\ne 0$ for all nonzero $l\in S_1$ (that is, $f$ essentially depends on three variables). Suppose indeed the contrary, and note that in this case $\blen f\le 3$, hence $\crk f\le 3$. Proposition~\ref{Br3} excludes $\crk f=3$. If $\crk f\le 2$, then there is some zero-dimensional scheme $Z'\subset\Ps S_1$, of degree at most two, such that $\Span{f}$ is in the span of $\nu\left(Z'\right)$, with $\nu:\Ps S_1\to\Ps S_4$ being the Veronese embedding. Then $h^1\left(\mathcal{I}_{Z\cup Z'}(4)\right)>0$, with $Z$ as in Example~\ref{Esempio}. Since $\deg Z\cup Z'\le 6$, $Z\cup Z'$ must lie on a line. But, by construction, $Z$ is not on a line.  Hence $f$ essentially depends on three variables.
\end{remark}

With the following result we reach the goal of the present section.

\begin{proposition}\label{8}
Assume $\dim S_1=3$ and let $f\in S_4$. We have $\Ork(f)\le 8$.
\end{proposition}
\begin{proof}
Let $X\subsetneq\Ps S_1$ be a (proper) closed subset. Let $\Sigma$ be the (necessarily finite) set of all $\Span{l}\in\Ps S^1$ such that the line $l=0$ is contained in $X$.

When $\crk f\ge 4$, Propositions~\ref{Predecomp} and~\ref{Decomp} give $\Ork(f,X)\le 8$. When $\crk f=3$, Proposition~\ref{Br3} gives $\Ork(f,X)=7<8$. Finally, if $\crk f\le 2$, then there is a zero-dimensional scheme $Z\subset\Ps S_1$ of degree at most two, such that $\Span{f}$ is in the span of $\nu(Z)$, with $\nu:\Ps S_1\to\Ps S_4$ being the Veronese embedding. Hence we can find distinct $\Span{l^0},\Span{l^1}\in S^1$, such that none of the lines $l^0=0$, $l^1=0$ is contained in $X$, and $l^0l^1=0$ contains $Z$. Therefore $\partial_{l^0l^1}f=0$, and Proposition~\ref{Duerette} gives $\Ork(f,X)\le 8$.\hfill $\square$
\end{proof}

Proposition~\ref{8} and Remark~\ref{EO} (or Example~\ref{SE}) together show that the maximum open rank for quaternary forms in essentially three variables is actually eight: in notation of \cite[Def.~2]{J}, $\Ork(3,4)=8$.

\begin{remark}
Proposition~\ref{8} and \cite[Theorems~4 and~5]{J} allow to improve \cite[Corollary~6]{J}, giving
\begin{equation}\label{UB}
\binom{n+d-2}{d-1}-\binom{n+d-6}{d-3}-\binom{n+d-7}{d-4}
\end{equation}
as an upper bound on open rank, hence on rank, for every $n\ge 3$, $d\ge 4$. Though \eqref{UB} is the best bound on open rank that we know to date, of course it is likely very far from being sharp. As a bound on rank, it is the best only for $(n,d)=(4,4)$ (likely far from being sharp, as well).
\end{remark}

\end{document}